\documentclass[10pt]{amsart}
\usepackage{enumerate,amsmath,amssymb,latexsym,
amsfonts, amsthm, amscd, MnSymbol}

\theoremstyle{plain}

\newtheorem{theorem}{Theorem}

\numberwithin{equation}{section}

\newcommand{\LL}{\mathbb{L}}
\newcommand{\da}{\vcrossing} 
\newcommand{\mtd}{\medtriangledown}

\addtocounter{section}{-1}

\begin{document}

\title {Projective space: lines and duality}

\date{}

\author[P.L. Robinson]{P.L. Robinson}

\address{Department of Mathematics \\ University of Florida \\ Gainesville FL 32611  USA }

\email[]{paulr@ufl.edu}

 \keywords{}

\begin{abstract}

We offer an axiomatic presentation of three-dimensional projective space that adopts the line as its fundamental element and renders automatic the principle of duality. 

\end{abstract}

\maketitle

\medbreak

\section{Introduction} 

\medbreak 

The traditional axiomatic approach to three-dimensional projective space takes point and line as its primary elements and incidence as its fundamental relation, in terms of which the notion of plane is derived as a secondary element. The principle of duality for projective space asserts that if in any theorem of the system we interchange the words point and plane (and make corresponding changes related to the mere terminology of incidence) then the result is also a theorem; for instance, the dual of the axiom that two points are joined by at least one line is the theorem that two planes meet in at least one line. Notice that line is a `self-dual' notion, whereas duality pairs point (a primary or fundamental notion) and plane (a secondary or derived notion). This circumstance entails some slight awkwardness of translation  in the proof of the duality principle: for example, if a statement refers to a pencil of planes (that is, to all the planes containing a given line) then its dual refers to all the points contained in a line (that is, to the line itself) and vice versa. See [VY] Chapter I and Theorem 11 in particular.  

\medbreak 

This asymmetry of the traditional axiomatic system, wherein the fundamental notion of point and the derived notion of plane are paired by duality, prompted us to consider setting up for projective space an axiomatic foundation that is better suited to duality. Ideally, the axioms of such a foundation should be manifestly self-dual, so that the principle of duality becomes automatic and requires no proof; the primary elements should be self-dual, the secondary elements paired by duality. The natural candidate for a self-dual primary element is of course the line; the natural candidate for a fundamental relation is abstract incidence of lines. Here we face something of a dilemma: in projective space, an incident pair of lines determines {\it both} a point {\it and} a plane; what we require is a device (expressed in terms of {\it lines}) for telling point and plane apart. Upon reflection, perhaps we should not expect that such a device is available: after all, if we are serious about duality then point and plane should be in a sense indistinguishable. As will be clear from our chosen axiomatic development, it is in fact possible to tell apart points and planes, in the sense that these secondary elements can be separated into two types; the question as to which type represents points and which type represents planes is then a matter of choice. 

\medbreak 

In Section 1 we present our self-dual axiom system, having line as its primary element and incidence as its fundamental relation; the key to the system lies in considering incident line pairs and their incident skew pairs. In Section 2 we introduce point and plane as secondary elements; we draw attention to the automatic duality principle and derive a number of theorems having implicit geometrical significance. In Section 3 we bring this geometrical significance to the fore and relate our axiom system to the traditional axiom system of Veblen and Young; for this traditional approach to projective space, we recommend the classic treatise [VY]. 

\medbreak 

\section{Lines}

\medbreak 

In this section we shall present our axioms for projective space, along with only as much of the theory as is necessary to frame the axioms. 

\medbreak 

Let $\LL$ be a nonempty set, whose elements we shall call {\it lines}. On $\LL$ we suppose a reflexive symmetric binary relation, which we denote by $\da$  and call {\it incidence}; this incidence relation $\da$ on $\LL$ has additional properties, encoded in AXIOM [1] - AXIOM [4] below. When $a, b \in \LL$ are incident we write $a \da b$; when $a, b \in \LL$ are not incident we call them {\it skew} and write $a \: | \: b$. More generally, we may refer to an arbitrary subset $S \subseteq \LL$ as being incident (respectively, skew) when it is so in the pairwise sense: that is, each pair of distinct lines $a, b \in S$ satisfies $a \da b$ (respectively, $a \: | \: b$). 

\medbreak 

When $S \subseteq \LL$ again, we write $S^{\da}$ for the set comprising all those lines that are incident to each line in $S$:  
$$S^{\da} = \{ \l \in \LL : (\forall s \in S) \ l\da s \}.$$
As usual, the familiar `Galois' properties hold: thus, if $S, S_1$ and $S_2$ are subsets of $\LL$ then 
$$S_1 \subseteq S_2 \Rightarrow S_2^{\da} \subseteq S_1^{\da}$$
$$S \subseteq S^{\da \da}$$ 
$$S^{\da \da \da} = S^{\da}.$$
When $\l_1, \dots , \l_n \in \LL$ we shall find it convenient to write 
$$[ \l_1 \dots  \l_n ] = \{ \l_1, \dots , \l_n \}^{\da}.$$

\medbreak 

After this preparation, we can begin to present the additional properties that we suppose satisfied by lines and incidence. These properties are embodied in four axioms, the first and most fundamental of which asserts somewhat more than the fact that each line is incident to others. 

\medbreak 

$\bullet$ AXIOM [1]: For each $\l \in \LL$ the set $\l^{\da} = [ \l ]$ contains a (pairwise) skew triple of lines. 

\medbreak 

Thus, if $\l \in \LL$ is arbitrary then there exist $x, y, z \in \l^{\da}$ such that $y\: | \:z, \: z\: | \:x, \: x\: | \:y$; in particular, $\l^{\da}$ is nonempty. 

\medbreak 

Our second axiom refers to two incident lines and their associated skew pairs; we separate this axiom into three parts for clarity. 

\medbreak 

$\bullet$ AXIOM [2.1]: If $a, b \in \LL$ are distinct and incident then $[ a b ] = \{a, b \}^{\da} = a^{\da} \cap b^{\da}$ contains skew pairs. 

\medbreak 

Equivalently, if $a, b$ is any incident pair of lines then $\{ a, b \}^{\da} \nsubseteq \{a, b\}^{\da \da}$. Notice that if $a = b$ then AXIOM [1] provides $[ a b ]$ with more than a skew pair. 

\medbreak 

$\bullet$ AXIOM [2.2]: If $a, b \in \LL$ are distinct and incident and if $z \in \{ a, b \}^{\da} \setminus \{a, b\}^{\da \da}$ then $[ a b z ]$ contains no skew pairs. 

\medbreak 

Otherwise said, if $a, b, z$ are as stated then 
$$x, y \in [ a b z ] \Rightarrow x \da y;$$
as we shall see later, this is essentially a (self-dual) version of the Pasch axiom familiar from traditional projective geometry. Of course, the existence of $z$ here is guaranteed by AXIOM [2.1]. 

\medbreak 

$\bullet$ AXIOM [2.3]: If $a, b \in \LL$ are distinct and incident and if $x, y \in [ a b ]$ is any skew pair then $a^{\da} \cap b^{\da} \subseteq x^{\da} \cup y^{\da}.$ 

\medbreak 

In other words, any line incident to $a$ and to $b$ is incident to $x$ or to $y$ (or to both). Note that in fact 
$$[ a b ] = [ a b x ] \cup [ a b y ].$$
Of course, the existence of $x, y$ here is guaranteed by AXIOM [2.1]. 

\medbreak 

Let $a, b \in \LL$ be two incident lines. Write 
$$\Sigma (a, b) = \{ a, b \}^{\da} \setminus \{a, b\}^{\da \da} = [ a b ] \setminus [ a b ]^{\da}$$ 
for the set of all lines that are one of a skew pair in $[ a b ]$. Henceforth, statements such as $c \in \Sigma (a, b)$ implicitly include the information that $a$ and $b$ are distinct incident lines. 

\begin{theorem} \label{Sigma}
For three lines $a, b, c \in \LL$ the conditions 
$$a \in \Sigma (b, c), \; b \in \Sigma (c, a), \; c \in \Sigma (a, b)$$
are equivalent. 
\end{theorem} 

\begin{proof} 
Suppose that $a \in \Sigma (b, c)$ but $c \notin \Sigma (a, b)$. As $c \in \{ a, b\}^{\da}$ already, it follows that $c \in  \{ a, b \}^{\da \da}$ whence $\{ a, b, c \} \subseteq \{ a, b \}^{\da \da}$ and therefore $\{ a, b, c \}^{\da} \supseteq \{ a, b \}^{\da \da \da} = \{ a, b \}^{\da}$ or 
$$[ a b ] \subseteq [ a b c ].$$ 
We have reached a contradiction: on the one hand, $[ a b ]$ contains skew pairs by AXIOM [2.1]; on the other hand, $[ a b c ]$ contains none by AXIOM [2.2]. 
\end{proof} 

The tension here between AXIOM [2.1] and AXIOM [2.2] appears again in Theorem \ref{coh} and Theorem \ref{exch}. 

\medbreak 

By a {\it triad} we shall mean three pairwise-incident lines $a, b, c$ that satisfy any (hence each) of the equivalent conditions in Theorem \ref{Sigma}. 

\medbreak 

Our third axiom refers to such triads of lines. 

\medbreak 

$\bullet$ AXIOM [3]: For each triad $a, b, c \in \LL$ there exists a triad $p, q, r \in \LL$ such that $[a b c] \cap [p q r] = \emptyset$. 

\medbreak 

In preparation for our last axiom, we require a couple of theorems and a definition. 

\medbreak 

Recall that incidence is reflexive and symmetric on $\LL$; on $\Sigma (a, b)$ more is true. 

\begin{theorem} \label{class}
Incidence restricts to $\Sigma (a, b)$ as an equivalence relation having precisely two equivalence clases. 
\end{theorem} 

\begin{proof} 
We need only verify transitivity and count the equivalence classes. For transitivity, let $p, q, r \in \Sigma (a, b)$: if $p \da q$ and $q \da r$ then $p, r \in [ a b q ]$ so that $p \da r$ by AXIOM [2.2].  The class count is also immediate: at least two by AXIOM [2.1]; at most two by AXIOM [2.3]. 
\end{proof} 

Let us write $\Sigma_{\upY} (a, b)$ and $\Sigma_{\mtd} (a, b)$ for the two incidence classes in $\Sigma (a, b)$: thus, $\Sigma (a, b)$ is the disjoint union 
$$\Sigma (a, b) = \Sigma_{\upY} (a, b) \cupdot \Sigma_{\mtd} (a, b).$$
Any two lines in $\Sigma_{\upY} (a, b)$ are incident and any two lines in $\Sigma_{\mtd} (a, b)$ are incident; any line from $\Sigma_{\upY} (a, b)$ and any line from $\Sigma_{\mtd} (a, b)$  are skew. At this stage, neither $\Sigma_{\upY} (a, b)$ nor $\Sigma_{\mtd} (a, b)$ is in any way preferred; we have simply named the incidence classes. 

\medbreak 

\begin{theorem} \label{z}
Let $c', c'' \in \Sigma (a, b)$: if $c' \da c''$ then $[ a b c' ] = [ a b c'' ]$. 
\end{theorem} 

\begin{proof} 
Let $l \in [ a b c' ]$: as $c'' \in [ a b c' ]$ too, AXIOM [2.2] implies that $l \da c''$; it follows that $l \in [ a b c'' ]$. This proves $[ a b c' ] \subseteq [ a b c'' ]$ while $[ a b c'' ] \subseteq [ a b c' ]$ symmetrically. 
\end{proof} 

On the basis of this theorem, we may well-define 
$$a \upY b = [ a b c_{\upY} ]$$
$$a \mtd b = [ a b c_{\mtd} ]$$
for any choice of $c_{\upY}$ in $\Sigma_{\upY} (a, b)$ and any choice of $c_{\mtd}$ in $\Sigma_{\mtd} (a, b)$. Observe that in these terms, AXIOM [2.3] yields 
$$[a b] = (a \upY b) \cup (a \mtd b).$$

\medbreak 

Our fourth and final axiom refers to two pairs of incident distinct lines and says somewhat more than that at least one line is incident to all of the lines in the two pairs.

\medbreak 

$\bullet$ AXIOM [4]: If $a, b$ and $p, q$ are pairs of distinct incident lines then 
$$(a \upY b) \cap (p \upY q) \neq \emptyset$$
$$(a \mtd b) \cap (p \mtd q) \neq \emptyset.$$

\medbreak 

Thus, if $c \in \Sigma_{\upY} (a, b)$ and $r \in \Sigma_{\upY} (p, q)$ then at least one line is incident to all of $a, b, c, p, q, r$ simultaneously; likewise when $\upY$ is replaced by $\mtd$.

\medbreak 

Notice that AXIOM [3] and AXIOM [4] together serve to coordinate the $\upY - \mtd$ labelling of incidence classes.  At the close of the next section, we observe that this labelling is coordinated across the whole of projective space; for now, we merely note coordination at each triad. 

\begin{theorem} \label{upYmtd}
If $a, b, c \in \LL$ is a triad then precisely one of the following sets of three equivalent conditions holds: \par
{\rm(}$\upY${\rm )} $a \in \Sigma_{\upY} (b, c), \; b \in \Sigma_{\upY} (c, a), \; c \in \Sigma_{\upY} (a, b);$ \par
{\rm (}$\mtd${\rm )}  $a \in \Sigma_{\mtd} (b, c), \; b \in \Sigma_{\mtd} (c, a), \; c \in \Sigma_{\mtd} (a, b).$ 
\end{theorem} 

\begin{proof} 
Let $p, q, r \in \LL$ be a triad with $[p q r] \cap [a b c] = \emptyset$ as provided by AXIOM [3]; of course, $r$ lies in either $\Sigma_{\mtd} (p, q)$ or $\Sigma_{\upY} (p, q)$. If $r \in \Sigma_{\mtd} (p, q)$ then $[p q r] = p \mtd q$ and AXIOM [4] implies that ($\upY$) holds; indeed, the alternative $c \in \Sigma_{\mtd} (a, b)$ would make $[a b c] = a \mtd b$ and the resulting $(a \mtd b) \cap (p \mtd q) = [a b c] \cap [p q r] = \emptyset$ would contradict AXIOM [4]. Likewise, if $r \in \Sigma_{\upY} (p, q)$ then $[p q r] = p \upY q$ and AXIOM [4] implies that ($\mtd$) holds. 
\end{proof} 

Remark: Thus, if $c \in \Sigma_{\upY} (a, b)$ then 
$$[a b c] = b \upY c = c \upY a = a \upY b$$
while if $c \in \Sigma_{\mtd} (a, b)$ then 
$$[a b c] = b \mtd c = c \mtd a = a \mtd b.$$

\medbreak 

In like manner, AXIOM [3] and AXIOM [4] also support the following. 

\begin{theorem} \label{pp}
If $a, b$ and $p, q$ are pairs of distinct incident lines then $a \mtd b \neq p \upY q$. 
\end{theorem} 

\begin{proof} 
Suppose $a \mtd b = p \upY q$; say $a \mtd b = [a b c]$ and $p \upY q = [p q r]$. AXIOM [3] furnishes a triad $x, y, z$ such that $[x y z]$ is disjoint from $[a b c] = [p q r]$. Now invoke AXIOM [4]: on the one hand, $[x y z] \cap (a \mtd b) = \emptyset$ implies that $[x y z] = x \upY y$; on the other hand, $[x y z] \cap (p \upY q) = \emptyset$ implies that $[x y z] = x \mtd y$. This contradicts the manifest inequality $x \upY y \neq x \mtd y$. 
\end{proof} 

This completes our list of axioms for projective space, founded only on the notions of line and incidence. In order to make contact with traditional approaches to projective space, we must define our notions of point and plane; we do this in terms of $\upY$ and $\mtd$ in the next section.

\medbreak 

\section{Points and Planes}

\medbreak 

Let $a, b \in \LL$ be an incident pair of distinct lines. We call $a \upY b$ the {\it point} determined by $a$ and $b$; we call $a \mtd b$ the {\it plane} determined by $a$ and $b$. Notice that $a \upY b$ and $a \mtd b$ are defined as subsets of $\LL$. Denote by $\LL_{\upY} \subseteq \mathcal{P} (\LL)$ the set of all such points and by $\LL_{\mtd} \subseteq \mathcal{P} (\LL)$ the set of all such planes. Notice that the sets $\LL_{\upY}$ and $\LL_{\mtd}$ are disjoint: points and planes are different, as we saw in Theorem \ref{pp}. 

\medbreak 

Before proceeding further, we note at once that the principle of duality holds automatically in the following form, for theorems derived from AXIOM [1] - AXIOM [4].  

\medbreak 

\noindent 
{\bf Duality Principle.} {\it The result of interchanging $\upY$ and $\mtd$ in any theorem is itself a theorem.}

\medbreak 

No proof is called for: $\upY$ and $\mtd$ only appear in AXIOM [4] where they appear symmetrically; the axioms are `self-dual'. 
In terms of our derived notions of point and plane, the principle of duality holds in its traditional form: the result of interchanging point and plane in any theorem will also be a theorem. 

\medbreak 

For the remainder of this section, we shall present and discuss results with reference to $\upY$ and $\mtd$ rather than with reference to points and planes. The implicit interpretations in terms of points and planes should at all times be clear; in any case, we shall explore sample point-plane interpretations carefully in the next section.  It is perhaps not entirely out of place here to mention that the symbol $\upY$ was chosen to suggest three (non-coplanar) lines meeting in a point and forming a {\it tripod}; the symbol $\mtd$ was chosen to represent three (non-copunctal) lines lying in a plane and forming a {\it trigon}. 

\begin{theorem}
If $l \in \LL$ then $\{ l \}^{\da \da} = \{ l \}$. 
\end{theorem}

\begin{proof} 
The containment $\supseteq$ is clear. Now, let $l \neq m \in \LL$; we must find $l' \da l$ such that $l' \: | \: m$. When $l \: | \: m$ we may take $l' = l$ itself. When $l \da m$ we offer two arguments. (i) Take a skew triad $\{ x, y, z \} \subseteq l^{\da}$ from AXIOM [1]: if $x \da m$ and $y \da m$ then $x, y \in \Sigma (l, m)$ is a skew pair and we may let $l' = z$; indeed, $z \da m$ would force $z \in x^{\da} \cup y^{\da}$ by AXIOM [2.3]. (ii) Choose $n \in \Sigma (l, m)$ by AXIOM [2.1] whence $m \in \Sigma (l, n)$ by Theorem \ref{Sigma}; then choose any partner $l' \in \Sigma (l, n)$ to make $l' , m$  a skew pair. 
\end{proof} 

\begin{theorem} \label{skew} 
Let the distinct lines $u, v, w$ be pairwise skew. If $m$ and $n$ are distinct lines in $[ u v w ]$ then $m$ and $n$ are skew. 
\end{theorem} 

\begin{proof} 
Suppose $m \da n$: note that 
$$\{ u, v, w \} \subseteq m^{\da} \cap n^{\da} \subseteq (m \upY n) \cup (m \mtd n)$$
where the second inclusion is provided by AXIOM [2.3]; it follows that either $m \upY n$ or $m \mtd n$ contains at least two of $u, v, w$ contrary to AXIOM [2.2]. 
\end{proof} 

Remark: This has the nature of a skew version of the Pasch property. Traditionally, $[ u v w ]$ is a {\it regulus}.  

\begin{theorem} \label{closed}
If $a, b, c \in \LL$ is a triad then $[ a b c ]^{\da} = [ a b c ]$. 
\end{theorem} 

\begin{proof} 
From $\{ a, b, c \} \subseteq \{ a, b, c \}^{\da}$ it follows that 
$$[ a b c ]^{\da} =  \{ a, b, c \}^{\da \da} \subseteq  \{ a, b, c \}^{\da} = [ a b c ].$$
In the opposite direction, AXIOM [2.2] ensures that each $l \in [ a b c ]$ meets each $m \in [ a b c ]$ so 
$$[ a b c ] \subseteq [ a b c ]^{\da}.$$
\end{proof} 

Thus, $(a \upY b)^{\da} = a \upY b$ and $(a \mtd b)^{\da} = a \mtd b$. 

\medbreak 

\begin{theorem} 
If $a$ and $b$ are two incident lines then $(a \upY b) \cap (a \mtd b) = [ a b ]^{\da}.$
\end{theorem} 

\begin{proof} 
The reverse inclusion $\supseteq$ follows from Theorem \ref{closed}: $a \upY b \subseteq [ a b ]$ so $[ a b ]^{\da} \subseteq (a \upY b)^{\da} = a \upY b$ and similiarly with $\upY$ relaced by $\mtd$. In the opposite direction, let $l \in (a \upY b) \cap (a \mtd b)$: if $m \in [ a b ]$ then $m \in (a \upY b) \cup (a \mtd b)$ by AXIOM [2.3] whence $l \da m$ by AXIOM [2.2]. 
\end{proof} 

Remark: It may be checked that $\Sigma_{\upY} (a, b) = (a \upY b) \setminus  [ a b ]^{\da}$ and $\Sigma_{\mtd} (a, b) = (a \mtd b) \setminus  [ a b ]^{\da}$. 

\begin{theorem} \label{coh} 
Let $[ a b c ] = [ p q r ]$. If $a, b, c$ is a triad then so is $p, q, r$ and conversely. 
\end{theorem} 

\begin{proof} 
Assume that $c \in \Sigma (a, b)$. By Theorem \ref{closed} it follows that 
$$[ a b c ] = [ a b c ]^{\da} = [ p q r ]^{\da} = \{ p, q, r \}^{\da \da} \supseteq \{ p, q, r \}$$
whence $p, q, r$ are (pairwise) incident by AXIOM [2.2]. Furthermore $p, q, r$ are distinct: any coincidence among them would by AXIOM [1] or AXIOM [2.1] force a skew pair in $[ a b c ]$ and so contradict AXIOM [2.2]. If $r \notin \Sigma (p, q)$ then $r \in \{ p, q \}^{\da \da}$ so that $\{ p, q, r \} \subseteq \{ p, q \}^{\da \da}$ and therefore 
$$[ p q ] = \{ p, q \}^{\da} = \{p, q \}^{ \da \da \da } \subseteq \{ p, q, r \}^{\da} = [ p q r ] = [ a b c ];$$
as $[ p q ]$ contains skew pairs by AXIOM [2.1] while $[ a b c ]$ does not by AXIOM [2.2] we have arrived at a contradiction, so $r \in \Sigma (p, q)$. The converse falls to symmetry. 
\end{proof} 

In an opposite direction we have the following result. 

\begin{theorem} \label{abcpqr} 
The conditions $p, q, r \in [ a b c ]$ and $a, b, c \in [ p q r ]$ are equivalent; when $a, b, c$ and $p, q, r$ are triads they imply $[ a b c ] = [ p q r ].$ 
\end{theorem} 

\begin{proof} 
The equivalence of $p, q, r \in [ a b c ]$ and $a, b, c \in [ p q r ]$ is evident from the symmetry of incidence. By Theorem \ref{closed},  $a, b, c$ a triad and $p, q, r \in [ a b c ]$ imply $[ p q r ] \supseteq [ a b c ]^{\da} = [ a b c ]$ while $p, q, r$ a triad and $a, b, c \in [ p q r ]$ imply $[ a b c ] \supseteq [ p q r ]^{\da} = [ p q r ].$
\end{proof} 

The next theorem exhibits what amounts to an exchange property. 

\begin{theorem} \label{exch} 
Let $a, b, c \in \LL$ be a triad. If $x, y \in [ a b c ]$ are distinct, then $\Sigma (x, y)$ contains at least one of $a, b, c$. 
\end{theorem} 

\begin{proof} 
Note at once that $x \da y$ by AXIOM [2.2]. Suppose none of $a, b, c$ lies in $\Sigma (x, y)$;  as $a, b, c \in \{ x, y \}^{\da}$ already, $\{ a, b, c \} \subseteq \{ x, y \}^{\da \da}$ and therefore 
$$[ a b c ] = \{ a, b, c \}^{\da} \supseteq \{ x, y \}^{\da \da \da} = \{ x, y \}^{\da} = [ x y ].$$
We have a contradiction: on the one hand, $[ x y ]$ contains skew pairs by AXIOM [2.1]; on the other hand, $[ a b c ]$ contains none by AXIOM [2.2]. 
\end{proof} 

This refines: if $[a b c] = a\mtd b$ then $\Sigma_{\mtd} (x, y)$ contains at least one of $a, b, c$. To see this, say $c \in \Sigma (x, y)$ by Theorem \ref{exch}; then $[x y c] = [a b c]$ by Theorem \ref{abcpqr} and $c \in \Sigma_{\mtd} (x, y)$ by an application of AXIOM [3] and AXIOM [4] as in the proof of Theorem \ref{upYmtd}.  

\medbreak Recall from Theorem \ref{pp} that points and planes are different. We can say more.  

\begin{theorem} 
$(a \mtd b) \cap (p \upY q)$ is not a singleton. 
\end{theorem} 

\begin{proof} 
Say $a \mtd b = [ a b c ]$ with $c \in \Sigma_{\mtd} (a, b)$ and $p \upY q = [ p q r ]$ with $r \in \Sigma_{\upY} (p, q)$. Assume that the stated intersection is nonempty: say $ l \in (a \mtd b) \cap (p \upY q)$; we shall show that the intersection then contains another line. We may suppose that each of the triples $a, b, c$ and $p, q, r$ already contains $l$ as follows. If $l \notin \{ a, b, c \}$ then (by Theorem \ref{exch} as subsequently refined) $\Sigma_{\mtd} (a, l)$ contains at least one of $a, b, c$ (other than $a$ of course); say it contains $b$ so that $l \in \Sigma_{\mtd} (a, b)$ and $a \mtd b = [ a b l ]$ by Theorem \ref{upYmtd}. Similarly, if $l \notin \{ p, q, r \}$ then we may replace $r$ (say) by $l$ so that $p \upY q = [ p q l ]$. Now choose $m \in \Sigma_{\upY} (a, b)$. AXIOM [4] ensures that $(a \upY b) \cap (p \upY q)$ is nonempty: say 
$$n \in [ a b m ] \cap [ p q l ] = \{ a, b, m, p, q, l \}^{\da}.$$ 
Plainly, $n \in \{ a, b, l \}^{\da} = [ a b l ] = a \mtd b$ and $n \in [ p q l ] = p \upY q$; plainly also $n \neq l$ because $n \in [ a b m ]$ and $m$ are incident but $l \in \Sigma_{\mtd} (a, b)$ and $m \in \Sigma_{\upY} (a, b)$ are skew.
\end{proof} 

\begin{theorem} \label{equal}
If $(a \mtd b) \cap (p \mtd q)$ contains more than one line then $a \mtd b = p \mtd q.$
\end{theorem} 

\begin{proof} 
Let $a \mtd b = [ a b c ]$ and $p \mtd q = [ p q r ]$ contain the distinct lines $x$ and $y$. By Theorem \ref{exch} (as subsequently refined) it follows that $\Sigma_{\mtd} (x, y) $ contains at least one of $a, b, c$ and at least one of $p, q, r$; say it contains $c$ and $r$ so that $[ a b c ] = [ x y c ]$ and $[ x y r ] = [ p q r ]$. Finally, by Theorem \ref{z} we deduce from $c, r \in \Sigma_{\mtd} (x, y)$ that $[ x y c ] = [ x y r ]$ and conclude that 
$$ [ a b c ] = [ x y c ] =  [ x y r ] = [ p q r ].$$
\end{proof} 

Remark: By duality, it is likewise true that 
$$\: | \: (a \upY b) \cap (p \upY q)\: | \: > 1 \Rightarrow a \upY b = p \upY q.$$

\medbreak 

We close this section with some comments on AXIOM [3] and AXIOM [4]. As we have already observed, these axioms guarantee that points and planes are different and they coordinate the $\upY - \mtd$ incidence partitions. Once the decision has been made for a specific incident line pair $(a, b)$ as to which incidence class should be $\Sigma_{\upY} (a, b)$ and which should be $\Sigma_{\mtd} (a, b)$ there is no freedom: $\Sigma_{\upY} (p, q)$ and $\Sigma_{\mtd} (p, q)$ are thereby determined for every incident line pair $(p, q)$.  To be explicit, let us fix (say) the point $Z : = a \upY b$ and take any other secondary element $\aleph$. The nature of this secondary element (point or plane) is forced: all we need do is consider the intersection $\aleph \cap Z$. As $\aleph$ is other than $Z$ there are two possibilities: if $\: | \: \aleph \cap Z\: | \: = 1$ then $\aleph$ is a point; if $\: | \: \aleph \cap Z\: | \: \neq 1$ then $\aleph$ is a plane. To put this another way: given any two secondary elements, we may at once decide whether they are of like type (two points or two planes) or of opposite types (one of each); the partition $\LL_{\upY} \cupdot \LL_{\mtd}$ is forced in all respects but the naming of its two parts. In these terms, `duality' amounts simply to a switching of the two names: the very same set of lines that was formerly called a point assumes the identity of a plane, and vice versa.

\medbreak

\section{Projective Space} 

\medbreak 

Having thus presented our axiom system and drawn from it an array of results, we now frame our discussion in more explicitly geometric terms. We also relate our axiom system to the traditional Veblen-Young axiom system, thereby addressing the question of consistency. 

\medbreak

The move toward traditional projective space will be accompanied by the introduction of some standard notation. It is customary to denote points by upper case Roman letters (such as $A, B, C$), lines by lower case Roman letters (such as $a, b, c$), planes by lower case Greek letters (such as $\alpha, \beta, \gamma$). Thus, when the lines $p', p'' \in \LL$ are incident and distinct, we may write $P = p' \upY p''$ for the point that they determine and $\pi = p' \mtd p''$ for the plane that they determine; also, we may write $P = [ p' p''p_{\upY} ]$ for any $p_{\upY} \in \Sigma_{\upY} (p', p'')$ and $\pi = [ p' p'' p_{\mtd} ]$ for any $p_{\mtd} \in \Sigma_{\mtd} (p', p'').$

\medbreak 

It is also customary to introduce terminological extensions of incidence. Let $p$ be a line, $P$ a point and $\pi$ a plane. We say that $P$ is on $p$ precisely when $p \in P$; we may also say that $p$ is on $P$ or that $p$ passes through $P$. We say that $\pi$ is on $p$ precisely when $p \in \pi$; we may also say that $p$ is on $\pi$ or that $p$ lies in $\pi$. We say that $P$ and $\pi$ are incident (or that the one is on the other) precisely when $P \cap \pi \neq \emptyset$: that is, some line is on both $P$ and $\pi$. We say that certain points or planes are collinear precisely when they are all on one line; elements are coplanar when they are all on one plane and copunctal when they are all on one point. We may use synonymous and similar expressions as convenient. 

\medbreak 

We illustrate the interpretation of our axiomatic development in projective geometric terms by a couple of examples; these theorems are chosen with a view to their usefulness in relating our axiom system to that of  Veblen and Young. 

\medbreak 

For our first example, let $A$ and $B$ be two points on the same plane $\pi$. As $A$ and $B$ are distinct, Theorem \ref{equal} for $\upY$ shows that they intersect in a singleton; say $A \cap B = \{ l \}$. To say that $A$ and $B$ lie on $\pi$ is to say that the intersection of each with $\pi$ is nonempty; say $a \in A \cap \pi$ and $b \in B \cap \pi$.  As we should expect, though it is not immediately apparent, the line $AB : = l$ lies in $\pi$. 

\begin{theorem} \label{lpi}
If $A$ and $B$ are distinct points on the plane $\pi$ then the line $AB$ lies in $\pi$. 
\end{theorem} 

\begin{proof} 
We continue to let $A \cap B = \{ l \}$, $a \in A \cap \pi$ and $b \in B \cap \pi$. If $a = b$ then from $A \ni a = b \in B$ it follows that $a = b = l$ and so $l \in \pi$; if $l = a$ or $l = b$ then $l \in \pi$ at once. We may therefore restrict attention to the case in which $l, a, b$ are distinct.  AXIOM [2.2] forces the lines $a, b \in \pi$ to be incident; as $\{ a, b \} \subseteq (a \mtd b) \cap \pi$ we deduce by Theorem \ref{equal} that $a \mtd b = \pi$. Now, assume $l \notin a \mtd b$. As $l \in [ a b ]$ already, it follows that $l \in a \upY b = : C$ in view of the reformulation of AXIOM [2.3] noted after Theorem \ref{z}. By Theorem \ref{equal} in its $\upY$ version, $l, a \in A \cap C$ and $l \neq a$ force $A = C$; likewise, $l \neq b$ in $B \cap C$ forces $B = C$. The deduction $A = B$ contradicts $A \neq B$ and so faults the assumption $l \notin a \mtd b$. Thus $l \in a \mtd b$ which with $a \mtd b = \pi$ places $l$ in $\pi$ as required. 
\end{proof} 

\medbreak

Dually, if two planes pass through a point then so does their common line. 

\medbreak 

For our second example, let $A, B, C$ be non-collinear points: thus, let $A, B, C \in \LL_{\upY}$ and $A \cap B \cap C = \emptyset$. Of course, the points $A, B, C$ are distinct: coincidence of two would render $A \cap B \cap C$ nonempty by AXIOM [4]. According to Theorem \ref{equal} in its $\upY$ version, $A \cap B$ is a singleton; say $A \cap B = \{ c \}$, with $B \cap C = \{ a \}$ and $C \cap A = \{ b \}$ likewise. Note that the lines $a, b, c$ are pairwise incident: for example, $a, b \in C$. Note also that $a, b, c$ are distinct: $a = b$ would imply $B \cap C \ni a = b \in C \cap A$ and thereby contradict the vacuity of $A \cap B \cap C$. 

\begin{theorem} \label{triangle}
If points $A, B, C$ are not collinear, with 
$$ B \cap C = \{ a \}, \: C \cap A = \{ b \}, \: A \cap B = \{ c \},$$
then 
$$ a \in \Sigma_{\mtd} (b, c), \: b \in \Sigma_{\mtd} (c, a), \: c \in \Sigma_{\mtd} (a, b),$$
and $A, B, C$ are coplanar, with [abc] the unique plane through all. 
\end{theorem} 

\begin{proof} 
As noted, $a, b, c$ are distinct and pairwise incident. In fact, $C \ni a, b$ is the point $a \upY b$ determined by $a$ and $b$; say $C = [ a b c_{\upY} ]$ with $c_{\upY} \in \Sigma_{\upY} (a, b)$. Note that $c \in A \cap B \subseteq [a b]$; as $A, B, C$ are not collinear, $c \notin C$ whence $c \: | \: c_{\upY}$. Thus $c \in \Sigma_{\mtd} (a, b)$; by Theorem \ref{upYmtd} also $ a \in \Sigma_{\mtd} (b, c)$ and $b \in \Sigma_{\mtd} (c, a)$. The plane 
$$\pi: = [ a b c ] = b \mtd c = c \mtd a = a \mtd b$$ 
has 
$$b, c \in \pi \cap A; \: c, a \in \pi \cap B; \: a, b \in \pi \cap C$$ 
and so passes through the points $A, B, C$ which are therefore coplanar. Finally, if $r \in \Sigma_{\mtd} (p, q)$ and the plane $[p q r]$ passes through $A, B, C$ then $[p q r]$ contains $a, b, c$ by Theorem \ref{lpi} and so $[p q r] = [a b c]$ by Theorem \ref{abcpqr}. 
\end{proof} 

Dually, if three planes do not share a line then they share a unique point. 

\medbreak 

As a last geometric example, AXIOM [3] and AXIOM [4] provide each triangle $ABC$ with a vertex $O$ so as to form a tetrahedron; dually, they provide each vertex with a triangle to form a tetrahedron. Explicitly, let $A B C$ be a triangle with (non-collinear) vertices $A, B, C$. As in Theorem \ref{triangle}, $ B \cap C = \{ a \}, \: C \cap A = \{ b \}, \: A \cap B = \{ c \}$ with $[ a b c ] = b \mtd c = c \mtd a = a \mtd b$. A triad $p, q, r$ as furnished by AXIOM [3] determines by AXIOM [4] a point $O = [p q r] = q \upY r = r \upY p = p \upY q$ not on the plane $[a b c]$. Now $O A B C$ is a tetrahedron, with vertex $O$ and triangular base $A B C$;  AXIOM [4] and Theorem \ref{equal} grant us the additional edges $\hat{a} = OA$, $\hat{b} = OB$, $\hat{c} = OC$. The set of lines $T = \{ a, b, c, \hat{a}, \hat{b}, \hat{c} \}$ is pairwise incident except for the pairs $a \: | \: \hat{a}$, $b \: | \: \hat{b}$, $c \: | \: \hat{c}$.  It may be checked that with this induced incidence,  $T$ yields a finite geometry that satisfies each of our axioms except AXIOM [1]: for instance, $[a b] = \{ a, b, c, \hat{c} \}$ and $[a b]^{\da} = \{ a, b \}$ so that $\Sigma (a, b)= \{ c, \hat{c} \}$; further, $\Sigma_{\mtd} (a, b) = \{ c \}$ and $\Sigma_{\upY} (a, b) = \{ \hat{c} \}$.

\medbreak 

We now demonstrate that in our axiom system, the traditional axioms for three-dimensional projective space are satisfied; to be specific, we shall demonstrate satisfaction of the basic axioms set forth in the classic treatise [VY] by Veblen and Young. There follows a list of the axioms for extension and alignment from Chapter I of [VY], each axiom being accompanied by a proof of its validity in our system. We follow the wording in [VY] exactly, except for two inessential departures: at (E3$'$), a closure axiom that caps the dimension at three;  and at (A3), the Veblen-Young version of the Pasch axiom.  

\medbreak 

(E0): {\it There are at least three points on every line}. 

\medbreak 

[AXIOM [1] furnishes each line $l$ with a skew triple $\{ x, y, z \} \subseteq l^{\da}$. The points $X = l \upY x$, $Y = l \upY y$, $Z = l \upY z$ on $l$ are distinct: observe that $z \in Z$ but $z \notin X \cup Y$.] 

\medbreak 

(E1): {\it There exists at least one line}. 

\medbreak 

[The set $\LL$ is nonempty!] 

\medbreak 

(E2): {\it All points are not on the same line}. 

\medbreak 

[Let $l$ be any line. AXIOM [1] furnishes a skew triple $x, y, z$ in $l^{\da}$ and then a skew triple $u, v, w$ in $z^{\da}$. We claim that $l \notin [ u v w ]$. To justify this claim, suppose to the contrary that $l \in [ u v w ]$: as $z \in [ u v w ]$ also, either $l = z$ or $l \: | \: z$ would follow by Theorem \ref{skew}; but $l \neq z$ and $l \da z$, so we have a contradiction. Thus, $l$ fails to be incident with at least one of $u, v, w$. Finally, if (say) $l \: | \: w$ then $z \upY w$ is a point not on $l$. Alternatively, simply pair $l$ with a distinct incident line $m$ (by AXIOM [1] again) and apply the following axiom to the plane $l \mtd m$.]

\medbreak 

(E3): {\it All points are not on the same plane}. 

\medbreak 

[This comes from AXIOM [3] by way of AXIOM [4]: for each plane $a \mtd b = [a b c]$ there exists a point $p \upY q = [p q r]$ such that $[a b c] \cap [p q r] = \emptyset$ whence $p \upY q$ does not lie on $a \mtd b$.] 

\medbreak 

(E3$'$): {\it Any two distinct planes have a line in common}. 

\medbreak 

[This is essentially AXIOM [4] for $\mtd$.] 

\medbreak 

(A1): {\it If $A$ and $B$ are distinct points, there is at least one line on both $A$ and $B$}. 

\medbreak 

[This is simply AXIOM [4] for $\upY$.] 

\medbreak 

(A2): {\it If $A$ and $B$ are distinct points, there is not more than one line on both $A$ and $B$}.

\medbreak 

[This is the contrapositive of Theorem \ref{equal} in its $\upY$ version.]

\medbreak 

(A3): {\it If $A, B, C$ are points not on the same line, the line joining distinct points $D$ (on the line $BC$) and $E$ (on the line $CA$) meets the line $AB$}. 

\medbreak 

[As in Theorem \ref{triangle}, $ B \cap C = \{ a \}, \: C \cap A = \{ b \}, \: A \cap B = \{ c \}$ where $c \in \Sigma_{\mtd} (a, b)$. We may suppose not only that $D$ and $E$ are distinct but also that they differ from $C$: indeed, if (say) $D = C$ then $DE = CE = b$ and this line certainly meets $c = AB$. As $C, D, E$ are distinct, the $\upY$ version of Theorem \ref{equal} tells us that $ C \cap D = \{ a \}, C \cap E = \{ b \},$ and (say) $D \cap E = \{ f \}$; in particular, $C \cap D \cap E = \emptyset$ so that $f \in D \cap E$ does not pass through $C$. From $a, f \in D$ and $b, f \in E$ we deduce that $f \in [ a b ]$ which with $f \notin C = a \upY b$ implies that $f \in a \mtd b$ by AXIOM [2.3] in the form noted after Theorem \ref{z}. Finally, if $f \neq c$ then $DE = f \in a \mtd b$ meets $AB = c \in a \mtd b$ in the point $F = f \upY c$ by AXIOM [2.2] while if $f = c$ then $DE = AB$ and there is nothing to do.]

\medbreak 

In the opposite direction, let us instead suppose given projective space as governed by the Veblen-Young axioms. Take for $\LL$ the set comprising all lines in this projective space; take for $\da$ the incidence of such lines. By definition, lines $a$ and $b$ are incident iff they have a (Veblen-Young) point in common. When $a$ and $b$ are distinct this point is unique and may be denoted by $a \wedge b$; when $a$ and $b$ are distinct they also lie in a unique (Veblen-Young) plane, which may be denoted by $a \vee b$. Moreover, if lines $a$ and $b$ lie in a plane then they pass through a point. All of this is evident from Chapter I of [VY]. 

\medbreak 

Now, to see that our AXIOM [1] is satisfied, let $l$ be any line; working within the framework of [VY] Chapter I we fashion three pairwise skew lines $x, y, z$ meeting $l$ as follows. Axiom (E0) places on $l$ distinct points $X$, $Y$, $Z$. According to (E2) there is a point $X'$ not on $l$; let $x = X X'$ be the line that exists by (A1) and is unique by (A2). The line $l$ and point $X'$ together determine a plane $\pi = X' l = X' X Y$. According to (E3) there is a point $Y'$ not on $\pi$; let $y = Y Y'$. Further, let $l' = X' Y'$. The lines $l$ and $l'$ are skew: if $l = X Y$ and $l' = X' Y'$ were to meet then $Y'$ would lie on the plane $X' X Y$; similarly, the lines $x$ and $y$ are skew. The line $l'$ contains $X'$, $Y'$, and at least one more point; say $Z'$. Finally, $z = Z Z'$ is skew to $x$ and $y$: otherwise, (A3) would force $l = ZX = ZY$ to meet $l' = Z' X' = Z' Y'$. 

\medbreak 

We leave the testing of AXIOM [2] - AXIOM [4] as exercises, with only the following brief comments. Let $a$ and $b$ be distinct, incident lines; as noted above, they determine a unique point $a \wedge b$ and a unique plane $a \vee b$. Write $\lsem a \wedge b \rsem$ for the set of all lines through the point $a \wedge b$ and $\lsem a \vee b \rsem$ for the set of all lines in the plane $a \vee b$. In classical terms, the intersection of $\lsem a \wedge b \rsem$ and $\lsem a \vee b \rsem$ is a {\it flat pencil}. The symmetric difference of $\lsem a \wedge b \rsem$ and $\lsem a \vee b \rsem$ is precisely the set $\Sigma (a, b)$ comprising all lines that are one of a skew pair in $[ a b ]$: 
$$\Sigma (a, b) = \lsem a \wedge b \rsem \Delta \lsem a \vee b \rsem.$$
Of course, it is reasonable (but not compulsory) to label the incidence classes in $\Sigma (a, b)$ as 
$$\Sigma_{\upY} (a, b) = \lsem a \wedge b \rsem \setminus \lsem a \vee b \rsem$$
and
$$\Sigma_{\mtd} (a, b) =  \lsem a \vee b \rsem \setminus \lsem a \wedge b \rsem$$
so that the point $a \upY b$ in our terminology coincides with the bundle $\lsem a \wedge b \rsem$ of lines through the point $a \wedge b$ in classical terminology; the contrary labelling would perhaps be perverse.  

\bigbreak

\begin{center} 
{\small R}{\footnotesize EFERENCES}
\end{center} 
\medbreak 

[VY] Oswald Veblen and John Wesley Young, {\it Projective Geometry}, Volume I, Ginn and Company, Boston (1910). 

\medbreak

\end{document}